\documentclass[11pt]{amsart}

\usepackage[colorlinks,linkcolor=black,anchorcolor=blue,citecolor=blue]{hyperref}

\usepackage{amscd,amsmath,amssymb,amsfonts,amsthm,mathrsfs}
\usepackage[a4paper,text={160mm,240mm},centering,headsep=5mm,footskip=10mm]{geometry}
\usepackage[all]{xy}

\usepackage{array}
\usepackage{color}

\numberwithin{equation}{section}
\newtheorem{theorem}{Theorem}[section]
\newtheorem{lemma}[theorem]{Lemma}
\newtheorem{keylemma}[theorem]{Key Lemma}
\newtheorem{proposition}[theorem]{Proposition}
\newtheorem{corollary}[theorem]{Corollary}

\newtheorem{definition}[theorem]{Definition}
\newtheorem{remark}[theorem]{Remark}

\newtheorem{notat}[theorem]{Notation}
\newtheorem*{notat*}{Notation}

\newtheorem{conj}[theorem]{Conjecture}

\newcommand{\Z}{\mathbb{Z}}
\newcommand{\Q}{\mathbb{Q}}
\newcommand{\N}{\mathbb{N}}
\renewcommand{\O}{\mathcal{O}}

\renewcommand{\r}{\rightarrow}

\newcommand{\Zar}{\text{Zar}}

\newcommand{\et}{\mathrm{\acute{e}t}}
\newcommand{\Spec}{\mathrm{Spec}}

\newcommand{\CH}{\mathrm{CH}}

\begin{document}
\title[Zero-cycles with coefficients in Milnor K-theory]{A restriction isomorphism for zero-cycles with coefficients in Milnor K-theory}
\author{Morten L\"uders}
\address{Fakult\"at f\"ur Mathematik, Universit\"at Regensburg, 93040 Regensburg, Germany}
\email{morten.lueders@mathematik.uni-regensburg.de}

\begin{abstract}
We prove a restriction isomorphism for Chow groups of zero-cycles with coefficients in Milnor K-theory for smooth projective schemes over excellent henselian discrete valuation rings. Furthermore, we study torsion subgroups of these groups over local and finite fields.
\end{abstract}

\thanks{The author is supported by the DFG through CRC 1085 \textit{Higher Invariants} (Universit\"at Regensburg).}
\maketitle
\tableofcontents

\setcounter{page}{1}
\setcounter{section}{0}
\pagenumbering{arabic}

\section*{Introduction}
Let $\mathcal{O}_K$ be an excellent henselian discrete valuation ring with quotient field $K$ and residue field $k=\mathcal{O}_K/\pi\mathcal{O}_K$. Let $X$ be a smooth and projective scheme over Spec$\mathcal{O}_K$ of fiber dimension $d$. Let $X_K$ denote the generic fiber and $X_0$ the reduced special fiber. By $X^{(p)}$ we denote the set of points of codimension $p$ in $X$.

We call the groups
$$\text{coker}[\bigoplus_{x\in X_0^{(d-1)}}K_{j-d+1}^Mk(x)\rightarrow \bigoplus_{x\in X_0^{(d)}}K_{j-d}^Mk(x)]$$
and 
$$H^1[\bigoplus_{x\in X^{(d-1)}}K_{j-d+1}^Mk(x)\rightarrow \bigoplus_{x\in X^{(d)}}K_{j-d}^Mk(x)\r \bigoplus_{x\in X^{(d+1)}}K_{j-d-1}^Mk(x)]$$
Chow groups of zero-cycles (resp. relative zero-cycles) with coefficients in Milnor K-theory. Here $H^1[A\r B\r C]:=\text{ker}[B\r C]/\text{im}[A\r B]$ for abelian groups $A,B,C$ and the complexes are the ones defined by Kato in \cite[Sec. 1]{Ka83} for arbitrary excellent finite dimensional schemes. The first complex coincides with the one defined by Rost in \cite{Ro96} for varieties over a field.
These cohomology groups are related to the higher Chow groups
$$\CH^j(X_0,j-d)$$
and
$$\CH^j(X,j-d).$$
In fact the groups $\text{coker}[\oplus_{x\in X_0^{(d-1)}}K_{j-d+1}^Mk(x)\rightarrow \oplus_{x\in X_0^{(d)}}K_{j-d}^Mk(x)]$ and $\CH^j(X_0,j-d)$ are isomorphic. This is well known, see for example \cite{Ak04'} (see also Section \ref{Some identifications}). The identification of $H^1[\oplus_{x\in X^{(d-1)}}K_{j-d+1}^Mk(x)\rightarrow \oplus_{x\in X^{(d)}}K_{j-d}^Mk(x)\r \oplus_{x\in X^{(d+1)}}K_{j-d-1}^Mk(x)]$ and $\CH^j(X,j-d)$ depends on the Gersten conjecture for a henselian DVR for higher Chow groups if we work integrally. However, which is sufficient for our purposes, they are isomorphic if considered with finite coefficients (see Section \ref{Some identifications}). We will therefore sometimes also call the groups $\CH^j(X_0,j-d)$ and $\CH^j(X,j-d)$ Chow groups of zero-cycles (resp. relative zero-cycles) with coefficients in Milnor K-theory.

In this article we study the restriction homomorphism on higher Chow groups 
$$res^{\CH}:\CH^j(X,2j-i)\xrightarrow{} \CH^j(X_0,2j-i)$$
for $i-j=d$. We recall the definition of the restriction homomorphism $res^{\CH}$ from \cite[Sec. 2]{Lu16}. It is defined to be the composition
$$\CH^j(X,2j-i)\xrightarrow{}\CH^j(X_K,2j-i)\xrightarrow{\cdot(-\pi)}\CH^{j+1}(X_K,2j-i+1)\xrightarrow{\partial}\CH^j(X_0,2j-i).$$
Here $\cdot(-\pi)$ is the product with a local parameter $-\pi\in \CH^1(K,1)=K^{\times}$ defined in \cite[Sec. 5]{Bl86} and $\partial$ is the boundary map coming from the localization sequence for higher Chow groups (see \cite{Le01}). We call the composition $$\CH^j(X_K,2j-i)\xrightarrow{\cdot(-\pi)}\CH^{j+1}(X_K,2j-i+1)\xrightarrow{\partial}\CH^j(X_0,2j-i)$$ a specialisation map and denote it by $sp^{\CH}_\pi$. One notes that $res^{\CH}$ is independent of the choice of $\pi$ whereas $sp^{\CH}_\pi$ depends on it. We denote higher Chow groups with coefficients in a ring $\Lambda$ by $\CH^j(X,j-d)_{\Lambda}$. From now on let $n$ be an integer greater than $1$ and invertible in $k$ and let $\Lambda=\Z/n\Z$. The main result of this article is the following:
\begin{theorem}[Thm. \ref{maintheoremintext}]\label{thmspchowmil}
The restriction map
$$res^{\CH}:\CH^j(X,j-d)_{\Lambda}\xrightarrow{} \CH^j(X_0,j-d)_{\Lambda}$$
is an isomorphism for all $j$.
\end{theorem}
This was conjectured in \cite[Conj. 10.3]{KEW16} by Kerz, Esnault and Wittenberg. More precisely, they conjecture that the corresponding restriction homomorphism for motivic cohomology
$res:H^{i,j}(X,\mathbb{Z}/n\mathbb{Z})\rightarrow H^{i,j}_{\mathrm{cdh}}(X_0,\mathbb{Z}/n\mathbb{Z})$
is an isomorphism for $i-j=d$. 
The case $j=d$ was first proved in \cite{SS10} assuming that $k$ is finite or separably closed and then generalised to arbitrary perfect residue fields in \cite{KEW16} using an idea of Bloch put forward in \cite{EWB16}. 

For $j=d+1$ and $k$ finite, Theorem \ref{thmspchowmil} also follows from the Kato conjectures. In fact, Jannsen and Saito observe in \cite{JS} that for $j=d+1$ and $k$ finite, the \'etale cycle class map
$$\rho_X^{j,j-d}:\CH^j(X,j-d)_{\Lambda} \r H^{d+j}_{\text{\'et}}(X,\Lambda(j))$$
fits into the exact sequence
\begin{equation*}
\begin{split}
...\r KH^{0}_{2+a}(X,\Z/n\Z)\r \CH^{d+1}(X,a)_{\Lambda}\r H_{\text{\'et}}^{2d+2-a}(X,\Lambda(d+1))\\
\r KH^{0}_{1+a}(X,\Z/n\Z)\r \CH^{d+1}(X,a-1)_{\Lambda}\r...,
\end{split}
\end{equation*}
where $KH^0_a(X,\Z/n\Z)$ denotes the homology of certain complexes $C^0_n(X)$ in degree $a$ defined by Kato. For more details see Section \ref{sectionkatocomplexes}. Note that we do not make any assumptions on $k$ in Theorem \ref{thmspchowmil}.

Theorem \ref{thmspchowmil} implies the following two well-known corollaries:
\begin{corollary}[Cor. \ref{cor1}]\label{cor1'}
Let $X$ be a smooth projective scheme over an excellent henselian discrete valuation ring $\O_K$ with finite or algebraically closed residue field. Then
$$\rho_X^{j,j-d}:\CH^j(X,j-d)_{\Lambda} \r H^{d+j}_{\et}(X,\Lambda(j))$$
is an isomorphism for all $j$.
\end{corollary}
\begin{corollary}[Cor. \ref{cor2}]
Let $X$ be as in Corollary \ref{cor1'} and let $X_K$ denote the generic fiber. Let the residue field of $K$ be finite (resp. separably closed). Then the groups
\begin{enumerate}
\item $\CH^j(X_K,j-d)_{\Lambda}$ are finite for all $j\geq 0$.
\item $\CH^j(X_K,j-d)_{\Lambda}=0$ for $j\geq d+3$ (resp. $j\geq d+2$).

\end{enumerate}
\end{corollary}

In the last two sections, we turn to torsion questions for Chow groups of zero-cycles with coefficients in Milnor K-theory. We show the following two propositions:
\begin{proposition}[Prop. \ref{finitenesstorsionjgeq1}]
Let $X_K$ be a smooth and proper scheme over a local field $K$ with ring of integers $\O_K$ and finite residue field $k$ of characteristic $p$. Assume that $X_K$ has good reduction over $\O_K$ and let $n>0$ be a natural number prime to $p$. Then for all $j\geq 1$ the groups $$\CH^{d+j}(X_K,j)[n]$$ are finite.
\end{proposition}

\begin{proposition}[Prop. \ref{propositionfinitenesstorsionfinitefield}]
Let $X$ be a smooth projective scheme of dimension $d$ over a finite field $k$ of characteristic $p>0$. Then the cycle class map
$$\rho:\CH^{d+1}(X,1)\{l\} \r  H^{2d+1}_{\et}(X,\Z_\ell(d+1))$$
is an isomorphism for $l$ a prime different from $p$, $\CH^{d+1}(X,1)\{p\}=0$ and the group
$$\CH^{d+1}(X,1)$$
is finite. 
\end{proposition}

\textit{Acknowledgements.} I would like to thank Moritz Kerz for many helpful comments and discussions. I would also like to thank Jean-Louis Colliot-Th\'el\`ene for helpful comments on Section \ref{secfiniteness}.

\section[Identifications]{Identification of zero-cycles with coefficients in Milnor K-theory}\label{Some identifications}
We start by recalling some basic facts about Milnor K-theory.
\begin{definition}
Let $k$ be a field. We define the $n$-th Milnor K-group $K^M_n(k)$ to be the quotient of $(k^{\times})^{\otimes n}$ by the Steinberg group, i.e. the subgroup of $(k^{\times})^{\otimes n}$ generated by elements of the form $a_1\otimes...\otimes a_n$ satisfying $a_i+a_j=1$ for some $1\leq i<j\leq n$. Elements of $K^M_n(k)$ are called symbols and the image of $a_1\otimes...\otimes a_n$ in $K^M_n(k)$ is denoted by $\{a_1,...,a_n\}$.
\end{definition}

One can easily see that in $K^M_2(k)$ the following relations hold:
$$\{x,-x\}=0 \text{   and   } \{x,x\}=\{x,-1\}$$
This implies the following lemma:
\begin{lemma}\label{lemmagenmilnork}
Let $K$ be a discrete valuation ring with ring of integers $A$, local parameter $\pi$ and residue field $k$. Then the abelian group $K^M_n(K)$ is generated by symbols of the form 
$$\{\pi,u_2,...,u_n\} \text{ and } \{u_1,u_2,...,u_n\}$$
with $u_i\in A^{\times}$ for $1\leq i \leq n$. 
\end{lemma}
Keeping the notation of Lemma \ref{lemmagenmilnork} and denoting the image of $u_i$ in $K^M_n(k)$ by $\bar u_i$, one can show that there exists a unique homomorphism
$$\partial:K^M_n(K)\r K^M_{n-1}(k)$$
satisfying
$$\partial(\{\pi,u_2,...,u_n\})= \{\bar u_2,...,\bar u_n\},$$
called the tame symbol, and a unique homomorphism
$$sp_\pi:K^M_n(K)\r K^M_n(k)$$
satisfying
$$sp_\pi(\{\pi^{i_1}u_1,\pi^{i_2}u_2,...,\pi^{i_n}u_n\})= \{\bar u_1,...,\bar u_n\},$$
called the specialisation map. Note that $sp_\pi$, unlike $\partial$, depends on the local parameter $\pi$ in $K$ and that $\partial(\{u_1,u_2,...,u_n\})=0$, if $u_i\in A^{\times}$ for all $1\leq i \leq n$.

We now return to the situation of the introduction: let $\mathcal{O}_K$ be an excellent henselian discrete valuation ring with quotient field $K$ and residue field $k=\mathcal{O}_K/\pi\mathcal{O}_K$ and always assume that $1/n\in k^\times$. 
Let $X$ be a smooth and projective scheme over Spec$\mathcal{O}_K$ of fiber dimension $d$. Let $X_K$ denote the generic fiber and $X_0$ the reduced special fiber. By $X_{(p)}$ we denote the set of points $x\in X$ such that $\text{dim}(\overline{\{x\}})=p$, where $\overline{\{x\}}$ denotes the closure of $x$ in $X$.

We use the following notation for Rost's Chow groups with coefficients in Milnor K-theory (see \cite[Sec. 5]{Ro96}):
$$C_p(X,m)=\bigoplus_{x\in X_{(p)}}(K_{m+p}^Mk(x))\otimes \mathbb{Z}/n\Z,$$
$$Z_p(X,m)=\text{ker}[\partial:C_p(X,m)\rightarrow C_{p-1}(X,m)],$$
$$A_p(X,m)=H_p(C_*(X,m))$$
and similarly for $X_0$ (resp. $X_K$) replacing $X$ by $X_0$ (resp. $X_K$).  Here the map $\partial$ is defined using the tame symbol on Milnor K-theory for discrete valuation rings defined above, the normalization and the norm map (see \cite{Ka86}). Furthermore, let
$$C_p^g(X,m)=\bigoplus_{x\in X^g_{(p)}}(K_{m+p}^Mk(x))\otimes \Z/n\Z$$
and
$$Z_p^g(X,m)=\text{ker}[\partial:C_p^g(X,m)\rightarrow C_{p-1}(X,m)]$$
be the corresponding groups supported on cycles in good position, i.e. the sum is taken over all $x\in X_{(p)}$ such that $\overline{\{x\}}$ is flat over $\O_K$.
Note that $C_k(X,-k)=Z_k(X)\otimes \Z/n\Z$, the group of $k$-cycles on $X$, i.e. the free abelian group generated by $k$-dimensional closed subschemes of $X$, tensored with $\Z/n\Z$.

Let now $\pi$ be a local parameter of $\O_K$. We define the restriction map
$$res_\pi: C_p(X,m)\rightarrow C_{p-1}(X_0,m+1)$$
to be the composition
$$res_\pi: C_p(X,m)\rightarrow C_{p-1}(X_K,m+1)\xrightarrow{\cdot\{-\pi\}} C_{p-1}(X_K,m+2)\xrightarrow{\partial}C_{p-1}(X_0,m+1),$$
where $C_p(X,m)\rightarrow C_{p-1}(X_K,m+1)$ is defined to be the identity on all elements supported on $X_{(p)}\setminus {X_0}_{(p)}$ and zero on ${X_0}_{(p)}$ and $\partial$ is the boundary map induced by the tame symbol. For the fact that this composition is compatible with the corresponding cycle complexes see \cite[Sec. 4]{Ro96}. For more details see \cite[Sec. 2]{Lu16}. We just recall that the restriction map $res_\pi$ depends on the choice of $\pi$ but the induced map $res: A_p(X,m)\rightarrow A_{p-1}(X_0,m+1)$ is independent of such a choice. From now on, we will fix a uniformizing parameter $\pi\in\O_K$ and write $res$ instead of $res_\pi$. 

We now prove the identifications of the two versions of Chow groups of zero-cycles with coefficients in Milnor K-theory stated in the introduction.
\begin{proposition}\label{identificationwithZero-cycles with coefficients in Milnor K-theory}
For all $j\geq 0$, there are the following isomorphisms:
\begin{enumerate}
\item $\CH^j(X_0,j-d)\cong A_0(X_0,j-d)$.
\item $\CH^j(X,j-d)_\Lambda\cong A_1(X,j-d-1)$.

\end{enumerate}
\end{proposition}
\begin{proof}
In \cite[Sec. 10]{Bl86}, Bloch proves the existence of the spectral sequence
\begin{equation}\label{spseq1}
^{\CH}E^{p,q}_1=\bigoplus_{x\in X_0^{(p)}}\CH^{r-p}(\text{Spec}k(x),-p-q)\Rightarrow \CH^r(X_0,-p-q).
\end{equation} 
Using the localization sequence for higher Chow groups for schemes over a regular noetherian base of dimension $\leq 1$ (see \cite{Le01}) and a limit argument, one also gets the existence of the spectral sequence 
\begin{equation}\label{spseq2}
^{\CH}E^{p,q}_1=\bigoplus_{x\in X^{(p)}}\CH^{r-p}(\text{Spec}k(x),-p-q)\Rightarrow \CH^r(X,-p-q).
\end{equation} 

Setting $j=r$, and using the isomorphism $\sigma:K^M_r(k(x))\xrightarrow{\cong} \CH^r(k(x),r)$ shown in \cite{NS89} and \cite{To92}, it follows from spectral sequence (\ref{spseq1}) that $\CH^j(X_0,j-d)$ is isomorphic to the cokernel of 
$$\partial:\bigoplus_{x\in X_0^{(d-1)}}K_{j-d+1}^Mk(x)\xrightarrow{} \bigoplus_{x\in X_0^{(d)}}K_{j-d}^Mk(x)$$
since $E^{\bullet, \geq -j+1}_1=0$ as $\CH^s(k(x),t)=0$ for $s>t$. Note that $\partial$ commutes with $\sigma$ and the differential coming from spectral sequence (\ref{spseq1}). This follows from \cite[Lem. 3.2]{GL00}. Lemma 3.2 of \textit{loc. cit.} is formulated for a discrete valuation ring over a field but the proof works exactly the same way for an arbitrary discrete valuation ring. We will use this in the proof of $(2)$, without further mentioning it, to identify the differentials coming from spectral sequence (\ref{spseq2}) with $\partial$ in the relative situation.

Similarly, we get from spectral sequence (\ref{spseq2}) that $\CH^j(X,j-d)_\Lambda$ is isomorphic to $A_1(X,j-d-1)$, noting that the map
$$d_1^{d,-j-1}:\bigoplus_{y\in X^{(d)}}\CH^{j-d}(k(y),j-d+1)_\Lambda\r \bigoplus_{x\in X^{(d+1)}}\CH^{j-d-1}(k(x),j-d)_\Lambda$$
is surjective for all $j$ (see Figure \ref{figure2}). This can be seen as follows: Let $x\in X^{(d+1)}$ and let $y$ be the generic point of a regular lift $Z$ of $x$ to $X$ which is flat over $\O_K$. By the Beilinson-Lichtenbaum conjecture (see Theorem \ref{Beillichenbaum}) there are isomorphisms 
$$\rho^{j-d,j-d+1}:\CH^{j-d}(k(y),j-d+1)_\Lambda\xrightarrow{\cong}H_{}^{j-d-1}(k(y),\Lambda(j-d))$$ 
and 
$$\rho^{j-d-1,j-d}:\CH^{j-d-1}(k(x),j-d)_\Lambda\xrightarrow{\cong}H_{}^{j-d-2}(k(x),\Lambda(j-d-1)).$$ 
The assertion now follows from the surjectivity of the map
$$\partial:H_{}^{j-d-1}(k(y),\Lambda(j-d))\r H_{}^{j-d-2}(k(x),\Lambda(j-d-1)).$$
and the identity $\partial\circ \rho^{j-d,j-d+1}=\rho^{j-d-1,j-d} \circ d_1^{d,-j-1}$. In fact, $\partial$ is also a differential in the coniveau spectral sequence 
\begin{equation*}
^{\text{\'et}}E_1^{p,q}(X,\Lambda(j))=\oplus_{x\in X^p}H^{q-p}(k(x),\Lambda(j-p))\Rightarrow H_{\text{\'et}}^{p+q}(X,\Lambda(j)).
\end{equation*} and there is a map of coniveau spectral sequences $\rho^{p,q}_X: ^{\CH}E^{p,q}_1\r ^{\text{\'et}}E_1^{p,q+2j}$ induced by the \'etale cycle class map (see also Section \ref{sectionkatocomplexes}).
The surjectivity can be shown as follows: Since $\O_{Z,x}$ is henselian, $H_{}^{j-d-2}(k(x),\Lambda(j-d-1))\cong H_{}^{j-d-2}(\O_{Z,x},\Lambda(j-d-1))$ by rigidity for \'etale cohomology. An element $\alpha\in H_{}^{j-d-2}(k(x),\Lambda(j-d-1))$ corresponding to an element $\alpha'\in H_{}^{j-d-2}(\O_{Z,x},\Lambda(j-d-1))$ lifts to an element $\alpha'\cup s\in H_{}^{j-d-1}(k(y),\Lambda(j-d))$, $s$ being a generator of the maximal ideal of $\O_{Z,x}$, with $\partial(\alpha'\cup s)=\alpha$ (see also \cite[Lem. 1.4 (2)]{Ka86}).
\end{proof}

\begin{figure}[b]
$$\begin{xy} 
  \xymatrix@-1.5em{
  &  & d & d+1 \\
 -j+1 & 0 &      0   &   0      \\
 -j & ... & \bigoplus_{x\in X^{(d)}}\CH^{j-d}(k(x),j-d) \ar[r]^{}  &      \bigoplus_{x\in X^{(d+1)}}\CH^{j-d-1}(k(x),j-d-1) \\
 -j-1 & ... & \bigoplus_{x\in X^{(d)}}\CH^{j-d}(k(x),j-d+1) \ar[r]^{}  &   \bigoplus_{x\in X^{(d+1)}}\CH^{j-d-1}(k(x),j-d)    \\
  & ... &      ...   &   ...           
  }
\end{xy} $$ 
\caption{Table of $^{\CH}E^{p,q}_1$ for $X/\O_K$.}
\label{figure2}
\end{figure}

\begin{remark}
\begin{enumerate}
\item Proposition \ref{identificationwithZero-cycles with coefficients in Milnor K-theory}(1) is proved by the same argument in \cite[Thm. 5.5]{Ak04'}. We recall the proof for the convenience of the reader. For similar identifications for $X_0$ with motivic cohomology (with modulus) see also \cite{RS}.
\item In order to show the isomorphism 
$$ \CH^j(X,j-d)\cong              
H^1(\bigoplus_{x\in X^{(d-1)}}K_{j-d+1}^Mk(x)\rightarrow \bigoplus_{x\in X^{(d)}}K_{j-d}^Mk(x) \r \bigoplus_{x\in X^{(d+1)}}K_{j-d-1}^Mk(x))$$
integrally, one would need to show the Gersten conjecture for a henselian DVR for higher Chow groups, i.e. the exactness of the sequence
$$0\r \CH^r(\mathrm{Spec}A,q)\r \CH^r(\mathrm{Spec}K,q)\r \CH^{r-1}(\mathrm{Spec}k,q-1)\r 0$$
for a henselian discrete valuation ring $A$ with field of fractions $K$ and residue field $k$.
\item Let $A$ be as in $(2)$. If $k$ is of characteristic $p>0$, then the sequence 
$$0\r \CH^r(A, \Z/p^r\Z, q)\r \CH^r(K, \Z/p^r\Z, q)\r \CH^{r-1}(k, \Z/p^r\Z, q-1)\r 0$$
is exact. This follows from the fact that in this case 
$\CH^{r-1}(\mathrm{Spec}k, \Z/p^r\Z, q-1)=0$ for $r\neq q$ by \cite[Thm. 1.1]{GL00} and that
$$\CH^r(K, \Z/p^r\Z, r)\cong K^M_r(K)/p^r\r K^M_{r-1}(k)/p^r\cong \CH^{r-1}(k, \Z/p^r\Z, r-1)$$
is surjective which implies that the long exact localization sequence
$$...\r \CH^r(A, \Z/p^r\Z, q)\r \CH^r(K, \Z/p^r\Z, q)\r \CH^{r-1}(k, \Z/p^r\Z, q-1)\r ...$$
splits (see also \cite[Cor. 4.3]{Ge04}).
\item If $k$ is finite, then $\CH^j(X_0,j-d)=0$ for $j>d+1$ since $K^M_2(k)=0$ for $n\geq 2$ (see also \cite[Cor. 7.1]{Ak04'} and \cite[Thm.7.1(1)]{KaS86}). If $K$ is a local field, then $\CH^j(X_K,j-d)_\Lambda=0$ for $j>d+2$ since $K^M_n(K)$ is uniquely divisible for $n\geq 3$ (see \cite[VI. 7.1]{We13}). 
\end{enumerate}
\end{remark}

\section{Relation with Kato complexes}\label{sectionkatocomplexes}
In this section we recall some facts about the Kato conjectures which we will need in the following sections.

Let $X$ be an excellent scheme. In \cite{Ka86}, Kato defines the following complexes:
\begin{multline*}
C^i_n(X): ...\r \bigoplus_{x\in X_a}H^{i+a+1}(k(x),\Z/n(i+a))\r ...\r 
\bigoplus_{x\in X_1}H^{i+2}(k(x),\Z/n(i+1))\\ \r \bigoplus_{x\in X_0}H^{i+1}(k(x),\Z/n(i))
\end{multline*}
Here the term $\oplus_{x\in X_a}H^{i+a+1}(k(x),\Z/n(i+a))$ is placed in degree $a$. We denote the homology of $C^i_n(X)$ in degree $a$ by $KH^i_a(X,\Z/n\Z)$. The groups $H^{i+a+1}(k(x),\Z/n(i+a))$ are the \'etale cohomology groups of $\Spec k(x)$ with coefficients in $\Z/n(i+a):=\mu^{\otimes i+a}_n$ if $n$ is invertible on $X$ and $\Z/n(i+a):=W_r\Omega^{i+a}_{X_1,\mathrm{log}}[-(i+a)]\oplus\Z/m(i)$ if $n=mp^r, (m,p)=1,$ is not invertible on $X$ and $X$ is smooth over a field of characteristic $p$.

The complex $C^0_n(X)$ for a proper smooth scheme $X$ over a finite field or the ring of integers in a ($1$-)local field is the subject of the study of the Kato conjectures. The Kato conjectures say the following and have been fully proved in case the coefficient characteristic is invertible on $X$ by Jannsen, Kerz and Saito (see \cite[Thm. 8.1]{KeS12}):

\begin{conj}(\cite[Conj. 0.3]{Ka86})\label{Katoconjfinitefield}
Let $X$ be a proper smooth scheme over a finite field. Then
$$KH^0_a(X,\Z/n\Z)=0 \quad \text{   for } a> 0.$$
\end{conj}
\begin{conj}(\cite[Conj. 5.1]{Ka86})\label{Katoconjlocalring}
Let $X$ be a regular scheme proper and flat over $\mathrm{Spec}(\O_k)$, where $\O_k$ is complete discrete valuation ring with finite residue field. Then
$$KH^0_a(X,\Z/n\Z)=0 \quad\text{   for } a\geq 0.$$
\end{conj}

Note that Conjecture \ref{Katoconjlocalring} is shown in \cite[Thm. 8.1]{KeS12} more generally for $\O_k$ a henselian discrete valuation ring with finite residue field. In \cite[Lem. 6.2]{JS}, Jannsen and Saito relate the complex $C^0_n(X)$ for a smooth scheme $X$ over a finite field to the \'etale cycle class map
$$\rho^{r,2r-s}_X:\CH^r(X,2r-s)_{\Lambda}\r H_{\text{\'et}}^{s}(X,\Lambda(r))$$
for $r=d$. More precisely, they show that there is an exact sequence 
\begin{equation}\label{JSexactsequence}
\begin{split}
...\r KH^0_{q+2}(X,\Z/n\Z)\r \CH^d(X,q)_{\Lambda}\r H_{\text{\'et}}^{2d-q}(X,\Lambda(d))\\
\r KH^0_{q+1}(X,\Z/n\Z)\r \CH^d(X,q-1)_{\Lambda}\r...
\end{split}
\end{equation}
This sequence is a tool to deduce finiteness results for Chow groups of higher zero cycles with finite coefficients from the Kato conjectures. The proof of the exactness of (\ref{JSexactsequence}) uses the coniveau spectral sequence for the domain and target of $\rho^{r,q}_X$ and the following theorem of Voevodsky:
\begin{theorem}(Beilinson-Lichtenbaum conjecture, see \cite{Vo11})\label{Beillichenbaum}
Let $X$ be a smooth scheme over a field. Then the \'etale cycle map
$$\rho^{r,2r-s}_X:\CH^r(X,2r-s)_{\Lambda}\r H_{\et}^{s}(X,\Lambda(r))$$
is an isomorphism for $s\leq r$.
\end{theorem}

We recall the following Proposition from \cite[Prop. 9.1]{KEW16}:

\begin{proposition}\label{katofinitefieldcase}
Let $X$ be a proper smooth scheme over a finite or algebraically closed field. Let $\Lambda=\Z/n\Z$ and $n$ be invertible on $X$. Then the \'etale cycle map
$$\rho^{j,j-d+a}_X:\CH^j(X,j-d+a)_{\Lambda}\r H_{\et}^{j+d-a}(X,\Lambda(j))$$
is an isomorphism for all $j\geq d$ and all $a$ except possibly if $k$ is finite, $j=d$ and $a=-1$. In particular the groups $\CH^j(X,j-d+a)_{\Lambda}$ are finite if $j\geq d, a\geq 0$. 
\end{proposition}
\begin{proof}
We consider the spectral sequences
$$ ^{\CH}E^{p,q}_1(X)=\oplus_{x\in X^{(p)}}\CH^{j-p}(\text{Spec}k(x),-p-q)_\Lambda\Rightarrow \CH^j(X,-p-q)_\Lambda$$
and
\begin{equation}\label{speseqniveaucohomology}
^{\text{\'et}}E_1^{p,q}(X,\Lambda(j))=\oplus_{x\in X^p}H^{q-p}(k(x),\Lambda(j-p))\Rightarrow H_{\text{\'et}}^{p+q}(X,\Lambda(j)).
\end{equation} 
The \'etale cycle class map $\rho^{p,q}_X$ induces a map of spectral sequences
$$\rho^{p,q}_X: ^{\CH}E^{p,q}_1\r ^{\text{\'et}}E_1^{p,q+2j}$$
which by Theorem \ref{Beillichenbaum} is an isomorphism for $q\leq -j$. By cohomological dimension, the difference between the two spectral sequences is given by
$$^{\text{\'et}}E_1^{\bullet,j+1}=C^{j-d}_n(X)$$
which is equal to the zero-complex if $j>d$ or if $k$ is an algebraically closed field. If $j=d$, then the complex $^{\text{\'et}}E_1^{\bullet,d+1}=C^{0}_n(X)$ is exact except for possibly the last term on the right due to Conjecture \ref{Katoconjfinitefield}. 
\end{proof}

We now turn to the arithmetic case (see also \cite[Sec. 9]{KeS12}).
\begin{proposition}\label{katochetale}
Let $X$ be a proper smooth scheme over a henselian discrete valuation ring $\O_K$ with finite residue field $k$. Let $\Lambda=\Z/n\Z$ and $n$ be invertible on $X$. Let $d$ be the relative dimension of $X$ over $\O_K$. Then for $j=d+1$, there is an exact sequence
\begin{equation}
\begin{split}
...\r KH^{0}_{2+a}(X,\Z/n\Z)\r \CH^{d+1}(X,a)_{\Lambda}\r H_{\et}^{2d+2-a}(X,\Lambda(d+1))\\
\r KH^{0}_{1+a}(X,\Z/n\Z)\r \CH^{d+1}(X,a-1)_{\Lambda}\r...
\end{split}
\end{equation}
and for $j>d+1$, there are isomorphisms
$$\CH^j(X,j-d+a)_{\Lambda}\r H_{\et}^{j+d-a}(X,\Lambda(j)).$$
\end{proposition}
\begin{proof}
We keep the notation of the proof of Proposition \ref{katofinitefieldcase}. Like there, we get a map of spectral sequences
$$\rho^{p,q}_X: ^{\CH}E^{p,q}_1\r ^{\text{\'et}}E_1^{p,q+2j}$$
which by Theorem \ref{Beillichenbaum} is an isomorphism for $q\leq -j$. The difference between the two spectral sequences is given by
$$^{\text{\'et}}E_1^{\bullet,j+1}=C^{j-d-1}_n(X)$$
if $j\geq d+1$ since all other rows vanish by cohomological dimension and $C^{j-d-1}_n(X)=0$ for $j\geq d+2$ again by cohomological dimension. This implies the proposition.
\end{proof}

\begin{remark}
For $X$ be a scheme over an excellent henselian discrete valuation ring with finite residue field and $j=d$ we are in the situation of \cite{SS10} which is more complex since there are two rows ($^{\et}E_1^{\bullet,j+1}=C^{-1}_n(X)$ and $^{\et}E_1^{\bullet,j+2}$) which might not be quasi-isomorphic to zero. In \cite{SS10}, Saito and Sato show that $KH^{-1}_{a}(X,\Q_n/\Z_n)=0$ for $a=2,3$. 
\end{remark}

We keep the notation of Proposition \ref{katochetale}. It follows from Conjecture \ref{Katoconjlocalring} that 
$$res^{\CH}:\CH^{d+1}(X,a)_{\Lambda}\r \CH^{d+1}(X_0,a)_{\Lambda}$$
is an isomorphism for all $a$. In the next section we generalise this result for $a=1$ to arbitrary residue fields. It remains an open problem if $res^{\CH}:\CH^{d+1}(X,a)_{\Lambda}\r \CH^{d+1}(X_0,a)_{\Lambda}$ is an isomorphism for arbitrary residue fields for all $a$. 
\begin{conj}
Let $X$ be as in the introduction. Then the restriction map $$res^{\CH}:\CH^i(X,j)_{\Lambda}\to \CH^i(X_0,j)_{\Lambda}$$ is an isomorphism for all $i\geq d+1$ and $j\geq 0$.
\end{conj}
This conjecture is a complement to the conjecture of Kerz, Esnault and Wittenberg in \cite[Sec. 10]{KEW16} saying that $res^\CH$ is an isomorphism for $i=d$ and $j\geq 0$.

\section{Main theorem}
We keep the notation of the introduction. In this section we prove Theorem \ref{thmspchowmil}. The strategy of the proof is inspired by \cite[Sec. 4,5]{KEW16}.
By the identifications of Proposition \ref{identificationwithZero-cycles with coefficients in Milnor K-theory}, proving Theorem \ref{thmspchowmil} comes down to studying the following commutative diagram:

$$\begin{xy} 
  \xymatrix{
  \bigoplus_{x\in X^{(d-1)}}K_{j-d+1}^Mk(x) \ar[r]^{sp} \ar[d]^{\partial}  &  \bigoplus_{x\in X_0^{(d-1)}}K_{j-d+1}^Mk(x) \ar[d]^{\partial}   \\
     \bigoplus_{x\in X^{(d)}}K_{j-d}^Mk(x) \ar[r]^{sp} \ar[d]^{\partial}   & \bigoplus_{x\in X_0^{(d)}}K_{j-d}^Mk(x)  \\
     \bigoplus_{x\in X^{(d+1)}}K_{j-d-1}^Mk(x) &
  }
\end{xy} $$  

We first note that one can lift zero-cycles with coefficients in Milnor K-theory from $X_0$ to one-cycles in good position with coefficients in Milnor K-theory.
\begin{proposition}\label{surjchmiln}
The restriction map
$$res:Z_1^g(X,j-d-1)\xrightarrow{} Z_0(X_0,j-d)$$
is surjective for all $j$. 
\end{proposition}
\begin{proof}
Let $\{\bar{u}_1,...,\bar{u}_{j-d}\}\in K_{j-d}^Mk(x)$ for some $x\in X_0^{(d)}$. Let $y\in X^{(d)}$ be the generic point of a lift $Z$ of $x$ which intersects $X_0$ transversally in $x$. Let $A$ be the stalk of $Z$ at $y$ and denote by $u_i\in A^{\times}$ a lift of $\bar{u}_i$ to the units of $A$. Then $res(\{u_1,...,u_{j-d}\})=\{\bar{u}_1,...,\bar{u}_{j-d}\}$ and $\partial(\{u_1,...,u_{j-d}\})=0$. 
\end{proof}

\begin{figure}
$$\begin{xy} 
  \xymatrix@-1.5em{
 j & ... & \bigoplus_{x\in X^{(1)}}H^{j-1}(k(x),\Lambda(j-1)) \ar[r]^{}  &      \bigoplus_{x\in X^{(2)}}H^{j-2}(k(x),\Lambda(j-2)) \\
 j-1 & ... & \bigoplus_{x\in X^{(1)}}H^{j-2}(k(x),\Lambda(j-1)) \ar[r]^{}  &      \bigoplus_{x\in X^{(2)}}H^{j-3}(k(x),\Lambda(j-2)) \\
  & ... &      ...   &   ...      \\
   2 & ...&    \bigoplus_{x\in X^{(1)}}H^1(k(x),\Lambda(j-1)) \ar[r]^{}    &  \bigoplus_{x\in X^{(2)}}H^0(k(x),\Lambda(j-2))  \\
   1 & ...&  \bigoplus_{x\in X^{(1)}}H^0(k(x),\Lambda(j-1)) & 0 \\
  0 & ... &      0   &   0      \\
   & 0 & 1 & 2
  }
\end{xy} $$ 
\caption{Table of $E^1_{p,q}(X,\Lambda)$ for $X/\O_K$ and $d=1$.}
\label{figure1}
\end{figure}

The following key lemma is the first step to construct an inverse map to $res:A_1(X,j-d-1)\xrightarrow{} A_0(X_0,j-d)$. The second step is done in the proof of Theorem \ref{maintheoremintext}.
\begin{keylemma}\label{keylemmazerowithkoeffmilnork}
\begin{enumerate}
\item Let 
$$\alpha\in \mathrm{ker}[Z_1^g(X,j-d-1)\r Z_0(X_0,j-d)].$$
Then $\alpha\equiv 0 \in A_1(X,j-d-1)$. 
\item There is a well-defined map $\phi: Z_0(X_0,j-d)\r A_1(X,j-d-1)$ making the diagram
$$\begin{xy} 
  \xymatrix{
   A_1(X,j-d-1)  &  \\
   Z_1^g(X,j-d-1)  \ar[r]_{res} \ar[u]  & Z_0(X_0,j-d) \ar[lu]_{\phi} 
  }
\end{xy} $$ 
commute.
\end{enumerate}
\end{keylemma}
\begin{proof}
(1) We start with the case of relative dimension $d= 1$. Let $\Lambda:=\mathbb{Z}/n$ and $\Lambda(q):=\mu_n^{\otimes q}$. We consider the coniveau spectral sequence
$$E_1^{p,q}(X,\Lambda(j))=\oplus_{x\in X^p}H^{q-p}(k(x),\Lambda(j-p))\Rightarrow H_{\text{\'et}}^{p+q}(X,\Lambda(j))$$
for $X$ and $X_0$ respectively (see Figure \ref{figure1} for $E^1_{p,q}(X,\Lambda)$ for $X/\O_K$ and $d=1$) and the norm residue isomorphism $K^M_n(k)/m\cong H^n(k,\mu_m^{\otimes n})$ to show that there are injective edge morphisms 
$$A_1(X,j-d-1)=E_2^{d,j}(X)\hookrightarrow H^{d+j}_{\text{\'et}}(X,\Lambda(j))$$ 
and 
$$A_0(X_0,j-d)=E_2^{d,j}(X_0)\hookrightarrow H^{d+j}_{\text{\'et}}(X_0,\Lambda(j)).$$ 
The injectivity in the second case, i.e. for $X_0$, is trivial since we just have two non-trivial columns for dimensional reasons. In the first case, we have three columns but $E_2^{2,j}(X,\Lambda(c))$ is equal to zero  for all $j$ since the map
$$\oplus_{x\in X^{(1)}}H^{j-2}(k(x),\Lambda(j-1)) \r     \oplus_{x\in X^{(2)}}H^{j-3}(k(x),\Lambda(j-2))$$
is surjective by the same arguments as in the proof of \ref{identificationwithZero-cycles with coefficients in Milnor K-theory}.
The restriction map induces a map between the respective spectral sequences for $X$ and $X_0$ and therefore a commutative diagram

$$\begin{xy} 
  \xymatrix{
  A_1(X,j-d-1) \ar[r]^{} \ar@{^{(}->}[d]_{} & A_0(X_0,j-d)  \ar@{^{(}->}[d]^{}  \\
     H^{d+j}_{\text{\'et}}(X,\Lambda(j)) \ar[r]^{\cong}  & H^{d+j}_{\text{\'et}}(X_0,\Lambda(j))  
  }
\end{xy} $$ 
whose lower horizontal morphism is an isomorphism by proper base change. It follows that $A_1(X,j-d-1)\rightarrow A_0(X_0,j-d)$ is injective. 

Let now $d>1$. We start with some reduction steps. Let 
$$\alpha\in \text{ker}[Z_1^g(X,j-d-1)\r Z_0(X_0,j-d)].$$  By definition, 
$$\alpha=\sum_{x\in X^g_{(1)}} \alpha_x\in  \text{ker}[res:\oplus_{x\in X^g_{(1)}}K_{j-d}^Mk(x) \r     \oplus_{x\in X_0^{(d)}}K_{j-d}^Mk(x)]$$
and $\alpha\in \text{ker}[\partial:\oplus_{x\in X^{(d)}}K_{j-d}^Mk(x) \r\oplus_{x\in X^{(d+1)}}K_{j-d-1}^Mk(x)]$ with $\alpha_x\in K_{j-d}^Mk(x)$.
We may assume that 
$$\alpha\in  \text{ker}[res:\oplus_{x\in X^g_{(1)}}K_{j-d}^Mk(x) \r  K_{j-d}^Mk(x_0)]$$ for some $x_0\in X_0^{(d)}.$

Let $y\in X^g_{(1)}$ be the generic point of a lift $Z_1$ of $x_0$ which intersects $X_0$ transversally in $x_0$.
We may now assume that 
$$\alpha=(\alpha_y,\alpha_z)\in  \text{ker}[res:K_{j-d}^Mk(y)\oplus K_{j-d}^Mk(z) \r  K_{j-d}^Mk(x_0)]$$ for some $z\in X^g_{(1)}.$ This follows from the fact that for every $z\in X^g_{(1)}$ which intersects $X_0$ in $x_0$, we can lift $res(\alpha_z)$ to an element $\alpha_y\in K_{j-d}^Mk(y)$ such that $\partial(\alpha_y)=\partial(\alpha_z)$. This can be seen as follows: Let $\alpha_y'\in K_{j-d}^Mk(y)$ be a lift of $res(\alpha_z)$. If $$\partial(\alpha_y')-\partial(\alpha_z)=\sum_{s\in S}\{\bar{u}_1^{(s)},...,\bar{u}_{j-d-1}^{(s)}\}\neq 0$$ for some $\bar{u}_i^{(s)}\in k(x_0)^\times, 1\leq i\leq j-d-1$ and some finite index set $S$, then choosing $u_i^{(s)}$ to be a lift of $\bar{u}_i^{(s)}$ to a unit in the discrete valuation ring of $k(y)$, we can set $$\alpha_y:=\alpha_y'-\sum_{s\in S}\{\pi,u_1^{(s)},...,u_{j-d}^{(s)}\}.$$ This has the required properties since $res(\{\pi,u_1^{(s)},...,u_{j-d}^{(s)}\})=0$ and $\partial(\{\pi,u_1^{(s)},...,u_{j-d}^{(s)}\})=\{\bar{u}_1^{(s)},...,\bar{u}_{j-d-1}^{(s)}\}$.

We now apply an idea of Bloch to reduce our situation to the case that $Z_1=\overline{\{y\}}$ intersects $X_0$ transversally and that $Z_2=\overline{\{z\}}$ is regular (see \cite[App.]{EWB16}). Let $\tilde Z_2$ be the normalisation of $Z_2$. Since $\mathcal{O}_K$ is excellent, $\tilde Z_2\r Z_2$ is finite and projective. This implies that there is an imbedding $\tilde Z_2\hookrightarrow X':=X\times_{\text{Spec}\mathcal{O}_K}\mathbb{P}^N$ such that the following diagram commutes:
$$\begin{xy} 
  \xymatrix{
  \tilde Z_2 \ar[r]^{} \ar[d]_{} & X'=X\times_{\text{Spec}\mathcal{O}_K}\mathbb{P}^N  \ar[d]^{pr_X}  \\
     Z_2 \ar[r]^{} \ar[d]_{} & X \ar[d]^{}  \\
     \text{Spec}\mathcal{O}_K \ar[r]^{=}     &  \text{Spec}\mathcal{O}_K
  }
\end{xy} $$
Let $(\tilde Z_2\cap X_0')_{\text{red}}= x'_0$ for $x'_0$ an integral zero-dimensional subscheme of $X'_0$. Let $\tilde Z_1$ be a regular lift of $x'$ in $Z_1\times \mathbb{P}^N\subset X'$ which has intersection number $1$ with $X_0'$. We denote the generic points of $\tilde Z_1$ and $\tilde Z_2$ by $y'$ and $z'$ respectively. Note first now that $\alpha_z\in K_{j-d}^Mk(z')$. Then, taking into account ramification, we can lift $res(\alpha_z)\in K_{j-d}^Mk(x'_0)$ to an element $\alpha_y'\in K_{j-d}^Mk(y')$ such that $(\alpha_y',\alpha_z)$ lies in the kernel of $res:\oplus_{x\in X'^{(d+N)}}K_{j-d+1}^Mk(x) \r \oplus_{x\in X_0'^{(d+N)}}K_{j-d}^Mk(x)$ as well as the kernel of $\partial:\oplus_{x\in X'^{(d+N)}}K_{j-d}^Mk(x) \r \oplus_{x\in X'^{(d+N-1)}}K_{j-d}^Mk(x)$ and such that, furthermore, we have that $pr_X((\alpha_y',\alpha_z))=(\alpha_y,\alpha_z)$.
It therefore remains to show that $(\alpha_y',\alpha_z)$ is in the image of the boundary map $$\partial:\oplus_{x\in X'^{(d+N+1)}}K_{j-d+1}^Mk(x) \r     \oplus_{x\in X'^{(d+N)}}K_{j-d}^Mk(x).$$ 

We show that the key lemma holds for 
$$\alpha=(\alpha_y,\alpha_z)\in  \text{ker}[res:K_{j-d}^Mk(y)\oplus K_{j-d}^Mk(z) \r  K_{j-d}^Mk(x_0)]$$
as above assuming that $Z_1=\overline{\{y\}}$ intersects $X_0$ transversally and that $Z_2=\overline{\{z\}}$ is regular by an induction on the relative dimension $d$ of $X$ over $\text{Spec}\mathcal{O}_K$.

Using a standard norm argument, we may assume that the residue field of $\text{Spec}\mathcal{O}_K$ is infinite. By standard Bertini arguments (cf. \cite[Sec. 4]{KEW16} or \cite[Lem. 2.6]{Lu16}), we can find smooth closed subschemes $S_1,S_2\subset X$ with the following properties:
\begin{enumerate}
         \item $S_1$ has fiber dimension one, $S_2$ has fiber dimension $d-1$.
         \item $S_1$ contains $Z_1$, $S_2$ contains $Z_2$.
         \item The intersection $S_1\cap S_2\cap X_0$ consist of reduced points.
      \end{enumerate}
Let $Z_3$ denote the component of $S_1\cap S_2$ containing $x_0$ and let $t$ denote its generic point. 
Then again there is an $\alpha_{t}\in K_{j-d}^Mk(t)$ such that
$res(\alpha_{t})=res(\alpha_{y})=-res(\alpha_{z})$ and $\partial(\alpha_{t})=\partial(\alpha_{y})=-\partial(\alpha_{z})$. 
Now by our induction assumption, both $(\alpha_y,\alpha_{t})$ and $(\alpha_x,\alpha_{t})$ map to zero in $A_1(X,j-d-1)$ and therefore so does $(\alpha_y,\alpha_z)$.

(2) This follows from (1) and Proposition \ref{surjchmiln}.
\end{proof}

\begin{theorem}\label{maintheoremintext}
The restriction map
$$res^{\CH}:\CH^j(X,j-d)_{\Lambda}\xrightarrow{} \CH^j(X_0,j-d)_{\Lambda}$$
is an isomorphism. 
\end{theorem}
\begin{proof}
By the identification of Proposition \ref{identificationwithZero-cycles with coefficients in Milnor K-theory}, it suffices to show that $res:A_1(X,j-d-1)\xrightarrow{} A_0(X_0,j-d)$ is an isomorphism.

We need to show that the map $\phi:Z_0(X_0,j-d)\r A_1(X,j-d-1)$ factorises through the group $A_0(X_0,j-d)$. In other words, we need to show that there is a $\overline{\phi}: A_0(X_0,j-d)\r A_1(X,j-d-1)$ such that the following diagram commutes:
$$\begin{xy} 
  \xymatrix{
    A_0(X_0,j-d) \ar[dr]^{\overline{\phi}}  &   \\
     Z_0(X_0,j-d) \ar[r]^{\phi} \ar[u]^{} & A_1(X,j-d-1)
  }
\end{xy} $$ 
Then $res\circ \overline{\phi}= id$ and since $\overline{\phi}$ is surjective, as we show below, the result follows.

Let $\alpha_0=(\alpha_0^1,...,\alpha_0^{j-d+1})\in C_1(X_0,j-d)=\oplus_{x\in X_0^{(d-1)}}K_{j-d+1}^Mk(x)\otimes \Z/n\Z$ be supported on some $x\in X_0^{(d-1)}$. As in the proof of \cite[Lem. 7.2]{SS10}, we can find a relative surface $Z\subset X$ containing $x$ which is regular at $x$  and such that $Z\cap X_0$ contains $\overline{\{x\}}$ with multiplicity $1$. Let $Z_0$ be the special fiber of $Z$ and denote by $\cup_{i\in I} Z_0^{(i)}\cup \overline{\{x\}}$ the union of the pairwise different irreducible components of $Z_0$. Here the irreducible components different from $\overline{\{x\}}$ are indexed by $I$. Let $z$ be the generic point of $Z$. Now as in the proof of \cite[Lem. 2.1]{Lu16}, we can for all $1\leq t\leq j-d+1$ find a lift $\alpha^t\in k(z)^{\times}$ of $\alpha_0^t$ which specialises to $\alpha_0^t$ in $k(x)^{\times}$ and to $1$ in $K(Z_0^{(i)})^{\times}$ for all $i\in I$. Let $\alpha=(\alpha^1,...,\alpha^{j-d+1})$. Then $\phi(\partial(\alpha_0))=\partial(\alpha)=0$ in $A_1(X,j-d-1)$ which implies the above factorisation.

The surjectivity of $\overline{\phi}$ follows from the surjectivity of $\phi$ which follows from key Lemma \ref{keylemmazerowithkoeffmilnork}: 
let $\alpha\in A_1(X,j-d-1)$. By arguments as in the last paragraph, one sees that $Z^g_1(X,j-d-1)$ generates $A_1(X,j-d-1)$. We may therefore assume that $\alpha\in Z^g_1(X,j-d-1)$. Let $\alpha_0$ be the restriction of $\alpha$ to $Z_0(X_0,j-d)$ and $\alpha'$ be a lift of $\alpha_0$ to $Z^g_1(X,j-d-1)$. Then, by Key Lemma \ref{keylemmazerowithkoeffmilnork}, we have that $\alpha\equiv \alpha'\in A_1(X,j-d-1)$.
\end{proof}

\section{Applications}\label{applications}
We list some applications of Proposition \ref{maintheoremintext}:

\begin{corollary}\label{cor1}
Let $X$ be a smooth projective scheme of relative dimension $d$ over an excellent henselian discrete valuation ring $\O_K$ with finite or algebraically closed residue field. Let $\Lambda=\Z/n\Z$ and $n$ be invertible on $X$. Then
$$\rho_X^{j,j-d}:\CH^j(X,j-d)_{\Lambda} \r H^{d+j}_{\et}(X,\Lambda(j))$$
is an isomorphism for all $j$. 
\end{corollary}
\begin{proof}
Consider the diagram
$$\begin{xy} 
  \xymatrix{
  \CH^j(X,j-d)_{\Lambda} \ar[r]^{\rho_X^{j,j-d}} \ar[d]_{\cong} & H^{d+j}_{\text{\'et}}(X,\Lambda(j)) \ar[d]^{\cong}   \\
     \CH^j(X_0,j-d)_{\Lambda} \ar[r]^{\cong}  & H^{d+j}_{\text{\'et}}(X_0,\Lambda(j)). 
  }
\end{xy} $$ 
The left vertical isomorphism follows from Proposition \ref{maintheoremintext}, the lower horizontal isomorphism from Proposition \ref{katofinitefieldcase} and the right vertical isomorphism from proper base change. This implies that $\rho_X^{j,j-d}$ is an isomorphism for all $j$. 
\end{proof}

\begin{corollary}\label{cor2}
Let $X$ be as in Corollary \ref{cor1} and let $X_K$ denote the generic fiber. Let the residue field of $K$ be finite (resp. separably closed). Then the groups
\begin{enumerate}
\item $\CH^j(X_K,j-d)_{\Lambda}$ are finite for all $j\geq 0$.
\item $\CH^j(X_K,j-d)_{\Lambda}=0$ for $j\geq d+3$ (resp. $j\geq d+2$).

\end{enumerate}
\end{corollary}
\begin{proof} Consider the localisation sequence
$$\CH^{j-1}(X_0,j-d)_{\Lambda}\xrightarrow{i_*} \CH^j(X,j-d)_{\Lambda}\xrightarrow{j^*} \CH^j(X_K,j-d)_{\Lambda}\xrightarrow{}\CH^{j-1}(X_0,j-d-1)_{\Lambda}.$$
By Corollary \ref{cor1} the groups $\CH^j(X,j-d)_{\Lambda}$ are finite for all $j$ and vanish for $j\geq d+2$. By Proposition \ref{katofinitefieldcase}, the groups $\CH^{j-1}(X_0,j-d-1)_{\Lambda}$ are finite for all $j$ and vanish for $j\geq d+3$ (resp. $j\geq d+2$). This implies that $\CH^j(X_K,j-d)_{\Lambda}$ is finite for all $j$, vanishes for $j\geq d+3$ (resp. $j\geq d+2$).
\end{proof}

\section{Torsion}\label{torsion}

In this section we prove finiteness theorems for torsion subgroups of some higher Chow groups of zero-cycles with coefficients in Milnor K-theory for smooth (projective) schemes over non-Archimedean local fields. These theorems generalise results of \cite{CSS82} (see also \cite{CSS83}) and \cite{Sz00}.

\begin{notat} For an abelian group $A$ we denote by $A[n]$ the kernel of the multiplication by $n$. For a prime $l$ we denote by $A\{l\}$ the $l$-primary torsion subgroups of $A$ and by $A_{\mathrm{tors}}$ the entire torsion subgroup of $A$. 

For $X$ a scheme we denote by $\mathcal{H}^q(\mu_n^{\otimes m})$ the Zariski sheaf associated to the presheaf $U\mapsto H^q_{\et}(U,\mu_n^{\otimes m})$.
\end{notat}

Let $X$ be a smooth variety over a field. Recall that by \cite{BO74} the Leray spectral sequence associated to the canonical morphism of sites $X_{\text{\'et}}\r X_{\text{Zar}}$
\begin{equation}\label{spseqchangeofsites}
E_2^{p,q}=H^p(X,\mathcal{H}^q(\mu_n^{\otimes m}))\Rightarrow H^{p+q}_{\text{\'et}}(X,\mu_n^{\otimes m})
\end{equation}
and the coniveau spectral sequence
\begin{equation}\label{spseqchangeofsites2}
E_1^{p,q}=\oplus_{x\in X^p}H^{q-p}(k(x),\mu_n^{\otimes m-p})\Rightarrow H^{p+q}_{\text{\'et}}(X,\mu_n^{\otimes m})
\end{equation}
agree from $E_2$ onwards and that therefore in particular
$$H^{p}(X,\mathcal{H}^q(\mu_n^{\otimes m}))=0 \text{ for } p>q.$$

\begin{proposition}\label{propfinitenessCH2(C,1)[n]}
Let $S$ be a smooth surface over a field $k$. Let $n>0$ be a natural number prime to the characteristic of $k$. Then for all $j\geq 0$ there is a surjection
$$\CH^{2+j}(S,j+1,\Z/n\Z)\twoheadrightarrow\CH^{2+j}(S,j)[n]$$
and $\CH^{2+j}(S,j+1,\Z/n\Z)$ is an extension of $H^1(S,\mathcal{K}^M_{2+j}/n)$ by a finite group. Furthermore, we have the following diagram:
$$\begin{xy} 
  \xymatrix{
   H^1(S,\mathcal{K}^M_{2+j}/n) \ar[d]^{\simeq}  & & \\
     H^1(S,\mathcal{H}^{2+j}(\mu_n^{\otimes 2+j})) \ar[r]^-{\simeq}  & E_\infty^{1,2+j} & \\
     & F^1H^{3+j} \ar@{->>}[u]^{} \ar@{^{(}->}[r]^{} & H^{3+j}_{\et}(S,\mu_n^{\otimes 2+j})
  }
\end{xy} $$ 
In particular, if $k$ is either separably closed, local with finite residue field or finite, then the groups $$\CH^{2+j}(S,j)[n]$$
are finite.
\end{proposition}
\begin{proof}
The surjectivity of $\CH^{2+j}(S,j+1,\Z/n\Z)\twoheadrightarrow\CH^{2+j}(S,j)[n]$ is clear. For the second statement consider the spectral sequence 
$$ ^{\CH}E^{p,q}_1(S)=\oplus_{x\in S^{(p)}}\CH^{2+j-p}(\text{Spec}k(x),-p-q)_\Lambda\Rightarrow \CH^{2+j}(S,-p-q)_\Lambda$$
and use Theorem \ref{Beillichenbaum} to show that $E_\infty^{1,-(2+j)}(S)\cong H^1(S,\mathcal{K}^M_{2+j}/n)$ and that $E_\infty^{2,-(3+j)}(S)$ is finite.
 
We now turn to the sequence of arrows in the diagram of the proposition. By the Bloch-Kato conjecture (see \cite{Vo11}) and the results recalled from \cite{BO74} at the beginning of this section, we have an isomorphism $$H^1(S,\mathcal{K}^M_{2+j}/n)\cong H^1(S,\mathcal{H}^{2+j}(\mu_n^{\otimes 2+j}))$$ in the Zariski topology. Since dim$(S)=2$, we have that $H^1(S,\mathcal{H}^{2+j}(\mu_n^{\otimes 2+j})) \cong E_2^{1,2+j}\cong E_\infty^{1,2+j}$ which is a quotient of $F^1H^{3+j}\subset H^{3+j}(S,\mu_n^{\otimes 2+j})$.
If $k$ is separably closed, then $H^{3+j}_{\text{\'et}}(S,\mu_n^{\otimes 2+j})$ is finite by \cite[Ch. XVI, Thm.  5.2]{SGA4}. This implies the finiteness for $k$ finite of local with finite residue field by the Hochschild-Serre spectral sequence.
\end{proof}

\begin{remark}
\begin{enumerate}
\item The case $j=0$ of Proposition \ref{propfinitenessCH2(C,1)[n]} was first shown in \cite{CSS82} and \cite{CSS83}. The case $j=1$ is shown in \cite{Sz00} assuming that $S$ is proper.
\item
Let $X$ be a smooth projective scheme of dimension $d$ over a $p$-adic field $K$. In \cite[Sec. 1, Sec. 5]{AS07}, Asakura and Saito show that neither the group $\CH^d(X)_{\mathrm{tors}}$ nor the group $\CH^{d+1}(X,1)_{\mathrm{tors}}$ may be expected to be finite.
\item
It would be interesting to have a conjecture on the expected structure of $\CH^j(X,j-d)^{\wedge p}$, the $p$-completion of $\CH^j(X,j-d)$, for $X$ a variety over a $p$-adic field for $j>0$. For $j=0$ see \cite[Sec. 1]{Co95}.

\end{enumerate}
\end{remark}
For proper smooth schemes one can generalise the above proposition for $j\geq 1$ to arbitrary dimension:
\begin{proposition}\label{finitenesstorsionjgeq1}
Let $X_K$ be a smooth and proper scheme of dimension $d$ over a local field $K$ with ring of integers $\O_K$ and finite residue field $k$ of characteristic $p$. Assume that $X_K$ has good reduction over $\O_K$ and let $n>0$ be a natural number prime to $p$. Then for all $j\geq 1$ the groups $$\CH^{d+j}(X_K,j)[n]$$ are finite.
\end{proposition}
\begin{proof}
For higher Chow groups we have the following exact sequence:
$$0\rightarrow \CH^s(X_K,t)/n\rightarrow \CH^s(X_K,t,\mathbb{Z}/n\mathbb{Z})\rightarrow \CH^s(X_K,t-1)[n]\rightarrow 0$$
In order to study $\CH^{d+j}(X_K,j)[n]$, one can therefore study the group $\CH^{d+j}(X_K,j+1,\mathbb{Z}/n\mathbb{Z})$. 
By Levine's localization sequence for higher Chow groups, $\CH^{d+j}(X_K,j+1,\mathbb{Z}/n\mathbb{Z})$ fits into the exact sequence 
$$\CH^{d+j}(X,j+1,\mathbb{Z}/n\mathbb{Z})\rightarrow \CH^{d+j}(X_K,j+1,\mathbb{Z}/n\mathbb{Z})\rightarrow \CH^{d+j-1}(X_k,j,\mathbb{Z}/n\mathbb{Z}).$$
By Proposition \ref{katofinitefieldcase}, Proposition \ref{katochetale} and the Kato conjectures, the groups $\CH^{d+j}(X,j+1,\mathbb{Z}/n\mathbb{Z})$ and $\CH^{d+j-1}(X_k,j,\mathbb{Z}/n\mathbb{Z})$ are finite if $j\geq 1$. This implies that $\CH^{d+j}(X_K,j)[n]$ is finite if $j\geq 1$.
\end{proof}

\section{A finiteness theorem}\label{secfiniteness}
In \cite{KaS86}, Kato and Saito prove the following theorem:
\begin{theorem}(\cite[Thm. 7.1]{KaS86})\label{KaSA71} Let $F$ be a number field, $\O_F$ the ring of integers of $F$ and $C$ an open subset of $\O_F$. Let $X$ be projective and integral over $C$ and $K$ be the function field of $X$. Let $d=\mathrm{dim}X$. Then the following statements hold:
\begin{enumerate} 
\item[(1)] If $n>d+1$, then the group $H^d_{\Sigma}(X_\mathrm{Nis},\mathcal{K}^M_{n+1}(\O_K,\mathcal{I}))$ vanishes for any non-zero ideal $\mathcal{I}$ of $\O_K$. 
\item[(2)] For a sufficiently small non-zero ideal $\mathcal{I}$ of $\O_K$, there exists a canonical isomorphism 
$$s_X: H^d_{\Sigma}(X_\mathrm{Nis},\mathcal{K}^M_{d+1}(\O_K,\mathcal{I}))\cong \mu(K).$$
\end{enumerate}
\end{theorem}
Here $\mathcal{K}^M_{d+1}(\O_K,\mathcal{I}):=\text{ker}[\mathcal{K}^M_{d+1}(\O_K)\r \mathcal{K}^M_{d+1}(\O_K/\mathcal{I})]$ and $\mu(K)$ is the group of all roots of $1$ in $K$. For the exact definition of $H^d_{\Sigma}(X_\mathrm{Nis},\mathcal{K}^M_{d+1}(\O_K,\mathcal{I}))$ see \cite[(1.4.1)]{KaS86}. We just note that $H^d_{\Sigma}(X_\mathrm{Nis},\mathcal{K}^M_{d+1}(\O_K,\mathcal{I}))=H^d_{}(X_\mathrm{Nis},\mathcal{K}^M_{d+1}(\O_K,\mathcal{I}))$ if $\text{ch}(K)\neq 0$. 

In \cite{Ak04}, Akhtar proves the following theorem using mixed K-groups:
\begin{theorem}
Let $X$ be a sooth projective variety over a finite field $k$ with structure morphism $\sigma:X\r \Spec(k)$, then the induced map $$\sigma_*: \CH^{d+1}(X,1)\r \CH^1(\Spec(k),1)\cong k^\times$$
is an isomorphism.
\end{theorem}

This last theorem implies in particular the finiteness of $\CH^{d+1}(X,1)$ for a smooth projective scheme $X$ of dimension $d$ over a finite field. In this section we give a different proof of this finiteness using the \'etale cycle class map, the Kato conjectures and the Weil conjectures proved by Deligne. We follow ideas of \cite{CSS82} and \cite{CSS83}. Note that $\CH^d(X)$ is studied in unramified class field theory and treated for example in \cite{CSS83}, \cite{KaS83} and \cite[Sec. 8]{Sz10}.

Let $X$ be a smooth scheme over a field $k$. Let $\Z(i)$ denote the motivic complex on the Zariski site of $X$ defined by Suslin and Voevodsky (see \cite{SV00}). For a ring $\Lambda$ we denote $\Z(i)\otimes \Lambda$ by $\Lambda(i)$. In the following let $l$ be a prime number. We recall that $\Z(i)$ has the following properties:
\begin{enumerate}
\item There is an isomorphism $H^n(X,\Z(i))\cong \CH^i(X,2i-n)$.
\item Let $\pi:X_{\text{\'et}}\r X_{\Zar}$ be the canonical map of sites. Let
\begin{equation*} \Z/l^n\Z(i):=
\begin{cases} 
\mu^{\otimes i}_{l^n} & l\neq p \\
v_n(i)[-i] & l=p
\end{cases}
\end{equation*}
in the derived category of \'etale sheaves on $X$.
There are quasi-isomorphisms 
$$\pi^*\Z(i)\otimes \Z/l^n\Z\xrightarrow{\cong} \Z/l^n\Z(i),$$
and
$$\Z(i)\otimes \Z/l^n\Z\xrightarrow{\cong} \tau_{\leq i}R\pi_*\Z/l^n\Z(i),$$
in the derived category of \'etale and Zariski sheaves on $X$ respectively. For $l\neq p$ these isomorphisms are shown in \cite{SV00} and \cite{SV96} respectively if $k$ is of characteristic $0$ and without assumption on $k$ in \cite{GL00} and \cite{GL01}. For $l=p$ they are shown in \cite[Thm. 8.5]{GL00}.
\end{enumerate}

We now note that for all $j\geq 0$ we have the following commutative diagram:
$$\begin{xy} 
  \xymatrix{
     H^{2d+j-1}(X,\Z/l^n\Z(d+j)) \ar[d]^{}  \ar@{->>}[r]^{} &  \CH^{d+j}(X,j)[l^n] \ar[d]^{} \\
      H^{2d+j-1}_{\text{\'et}}(X,\Z/l^n\Z(d+j)) \ar[r]^{}  &  H^{2d+j}_{\text{\'et}}(X,\Z/l^m\Z(d+j))
  }
\end{xy} $$ 
Let us recall the construction (cf. \cite[Sec. 3]{Sz00}, where the commutativity of the above diagram is shown in detail, and \cite[Sec. 8]{Sz10}):
\begin{enumerate}
\item The upper horizontal map comes from the distinguished triangle
$$\Z(i)\xrightarrow{l^n} \Z(i)\r \Z/l^n\Z(i)\r \Z(i)[1].$$
\item The lower horizontal map comes from the exact sequence
$$0\r \mu^{\otimes d+j}_{l^m}\rightarrow{} \mu^{\otimes d+j}_{l^{mn}}\r \mu^{\otimes d+j}_{l^n}\r 0$$ 
on $X_{\text{\'et}}$ inducing the distinguished triangle
$$R\pi_*\mu^{\otimes d+j}_{l^m}\r R\pi_*\mu^{\otimes d+j}_{l^{mn}}\r R\pi_*\mu^{\otimes d+j}_{l^n} \r R\pi_*\mu^{\otimes d+j}_{l^m}[1].$$
in the $\text{D}(X_{\text{Zar}})$.
\item The two vertical maps are induced by the change of sites map $\pi:X_{\text{\'et}}\r X_{\text{Zar}}$, i.e. 
\begin{equation*} \begin{split}
H^{2d+j-1}(X,\Z/l^n\Z(d+j))\xrightarrow{\simeq} H^{2d+j-1}(X,\tau_{\leq d+j}R\pi_*\Z/l^n\Z(d+j))\\ \r  H^{2d+j-1}(X,R\pi_*\Z/l^n\Z(d+j))\xrightarrow{\simeq} H^{2d+j-1}_{\text{\'et}}(X,\Z/l^n\Z(d+j))
\end{split} \end{equation*} 
and 
$$H^{2d+j}(X,\Z(d+j))\r H^{2d+j}(X,\tau_{\leq d+j}R\pi_*\Z/l^n\Z(d+j))\r  H^{2d+j}_{\text{\'et}}(X,\Z/l^n\Z(d+j)).$$
\end{enumerate}

Taking the colimit over $n$ and the limit over $m$ in the above commutative diagram, we get the following commutative diagram:
\begin{equation}\label{commutativediagramprop2.6.2}
\begin{xy} 
  \xymatrix{
     H^{2d+j-1}(X,\Q_\ell/\Z_\ell(d+j)) \ar[d]^{}  \ar@{->>}[r]^{} &  \CH^{d+j}(X,j)\{l\} \ar[d]^{} \\
      H^{2d+j-1}_{\text{\'et}}(X,\Q_\ell/\Z_\ell(d+j)) \ar[r]^{}  &  H^{2d+j}_{\text{\'et}}(X,\Z_\ell(d+j))
  }
\end{xy} 
\end{equation}

\begin{proposition}\label{propositionfinitenesstorsionfinitefield}
Let $X$ be a smooth projective scheme of dimension $d$ over a finite field $k$ of characteristic $p>0$. Then the cycle class map
$$\rho:\CH^{d+1}(X,1)\{l\} \r  H^{2d+1}_{\et}(X,\Z_\ell(d+1))$$
is an isomorphism for $l$ a prime different from $p$, $\CH^{d+1}(X,1)\{p\}=0$ and the group
$$\CH^{d+1}(X,1)$$
is finite. 
\end{proposition}
\begin{proof}
We first note that $\CH^{d+1}(X,1)=\CH^{d+1}(X,1)_{\text{tors}}$ since $$\CH^{d+1}(X,1)\cong \mathrm{coker}[\bigoplus_{x\in X^{(d-1)}}K_{2}^Mk(x)\rightarrow \bigoplus_{x\in X^{(d)}}K_{1}^Mk(x)]$$ and since $K_{1}^Mk(x)=k(x)^\times$ is a finite group for $x\in X^{(d)}$. For the last statement it therefore suffices to show the finiteness of $\CH^{d+1}(X,1)_{\text{tors}}$. 

Let us first show that $\CH^{d+1}(X,1)$ does not contain any $p$-torsion.
Since $\mathcal{K}^M_{X,*}$ is $p$-torsion free by \cite{Iz91}, we get an exact sequence
$$0\r \mathcal{K}^M_{X,*}\xrightarrow{p^n} \mathcal{K}^M_{X,*}\r \mathcal{K}^M_{X,*}/p^n\r 0.$$
This induces a surjection 
$$H^{d-1}(X,\mathcal{K}^M_{X,d+1}/p^n) \twoheadrightarrow H^{d}(X,\mathcal{K}^M_{X,d+1})[p^n]\cong \CH^{d+1}(X,1)[p^n].$$
Since $H^{d-1}(X,\mathcal{K}^M_{X,d+1}/p^n)\cong H^{d-1}(X,v_n(d+1))=0$ by the Bloch-Kato-Gabber theorem, the statement follows.

Let $l$ be a prime different from $p$. For $j=1$, diagram (\ref{commutativediagramprop2.6.2}) takes the form
\begin{equation*}
\begin{xy} 
  \xymatrix{
     H^{2d}(X,\Q_\ell/\Z_\ell(d+1)) \ar[d]^{\cong}  \ar@{->>}[r]^{} &  \CH^{d+1}(X,1)\{l\} \ar[d]^{\rho} \\
      H^{2d}_{\text{\'et}}(X,\Q_\ell/\Z_\ell(d+1)) \ar[r]^{\cong}  &  H^{2d+1}_{\text{\'et}}(X,\Z_\ell(d+1))
  }
\end{xy} 
\end{equation*}
The left vertical morphism is an isomorphism by Proposition \ref{katofinitefieldcase} ($j=d+1$ and $a=1$). That the lower horizontal map is an isomorphism follows from the vanishing of $H^{2d}_{\text{\'et}}(X,\Q_\ell(d+1))$ and $H^{2d+1}_{\text{\'et}}(X,\Q_\ell(d+1))$ which follows from the fact that the groups 
$$H^{i}_{\text{\'et}}(X,\Z_\ell(r))$$
are torsion for $i\neq 2r,2r+1$. This follows from the Weil conjectures (see \cite[Sec. 2]{CSS83}). In particular $\rho$ is an isomorphism. 
The finiteness of $\CH^{d+1}(X,1)_{\text{tors}}$ now follows from the above diagram and the fact that for a smooth projective scheme $X$ over a finite field $k$, the groups 
$$H^{i}_{\text{\'et}}(X,\Q_\ell/\Z_\ell(r))$$
are finite for $i\neq 2r,2r+1$ and zero for almost all $\ell$ (see \cite[Thm. 2]{CSS82} and \cite[Sec. 2]{CSS83})
and the fact that $\CH^{d+1}(X,1)$ does not contain any $p$-torsion which we showed in the beginning. 
\end{proof}

\begin{remark}
Two remarks are in order:
\begin{enumerate}
\item The injectivity of $\rho$, implying the finiteness of $\CH^{d+1}(X,1)$, may be reduced to curves as follows:
let $C$ be a smooth curve over $k$ and $l$ be a prime number prime to $p$. In this case we have the following commutative diagram by the the same arguments as in the proof of Proposition \ref{propositionfinitenesstorsionfinitefield}:
$$\begin{xy} 
  \xymatrix{
     H^{2}(C,\Q_\ell/\Z_\ell(2)) \ar[d]^{\cong}  \ar@{->>}[r]^{} &  \CH^{2}(C,1)\{l\} \ar[d]^{} \\
      H^{2}_{\et}(C,\Q_\ell/\Z_\ell(2)) \ar[r]^{\cong}  &  H^{3}_{\et}(C,\Z_\ell(2))
  }
\end{xy} $$ 
In particular, the map $\CH^{2}(C,1)\{l\}   \r  H^{3}_{\et}(C,\Q_\ell/\Z_\ell(2))$ is an isomorphism. 

We now show that the map $\CH^{d+1}(X,1)\{l\} \r  H^{2d+1}_{\et}(X,\Z_\ell(d+1))$ is injective. Let $\alpha\in \mathrm{ker}[\CH^{d+1}(X,1)\{l\} \r  H^{2d+1}_{\et}(X,\Z_\ell(d+1))]$. By the Bertini theorem of Poonen (see \cite{Po08}) and noting that 
$$\CH^{d+1}(X,1)\cong \mathrm{coker}[\bigoplus_{x\in X^{(d-1)}}K_{2}^Mk(x)\rightarrow \bigoplus_{x\in X^{(d)}}K_{1}^Mk(x)],$$ 
we can find a smooth curve $C$ containing the support of $\alpha$. Since $\CH^{2}(C,1)$ is torsion, $\alpha\in \CH^{2}(C,1)\{lm\}$ for $m$ prime to $l$. Furthermore $m$ is prime to $p$ since  $\CH^{2}(C,1)$ is $p$-torsion free. The injectivity now follows from the following commutative diagram: 
$$\begin{xy} 
  \xymatrix{
      \CH^{2}(C,1)\{lm\} \ar[d]^{}  \ar@{^{(}->}[r]^{} &    H^{3}_{\et}(C,\Z_{\ell m}(2)) \ar[d]^{\simeq} \\
      \CH^{d+1}(X,1)\{lm\} \ar[r]^{}  &    H^{2d+1}_{\et}(X,\Z_{\ell m}(d+1)) \\
      \CH^{d+1}(X,1)\{l\} \ar[r]^{} \ar@{^{(}->}[u]^{} &    H^{2d+1}_{\et}(X,\Z_\ell(d+1)) \ar@{^{(}->}[u]^{}
  }
\end{xy} $$ 
Note that the upper horizontal map is injective since $H^{i}_{\text{\'et}}(X,\Z_{\ell m}(r))\cong H^{i}_{\text{\'et}}(X,\Z_\ell(r))\times H^{i}_{\text{\'et}}(X,\Z_m(r))$.
The isomorphism on the right follows from the weak Lefschetz theorem (see \cite[Thm. 7.1]{Mi80}) and the Hochschild-Serre spectral sequence (cf. \cite[p. 793]{CSS83}).
\item Let $X$ be a smooth projective scheme of dimension $d$ over a finite field $k$ of characteristic $p>0$. For $j=0$, diagram (\ref{commutativediagramprop2.6.2}) takes the form
\begin{equation*}
\begin{xy} 
  \xymatrix{
     H^{2d-1}(X,\Q_\ell/\Z_\ell(d)) \ar[d]^{\cong}  \ar@{->>}[r]^{} &  \CH^{d}(X)\{l\} \ar[d]^{} \\
      H^{2d-1}_{\et}(X,\Q_\ell/\Z_\ell(d)) \ar@{^{(}->}[r]^{}  &  H^{2d}_{\et}(X,\Z_\ell(d))
  }
\end{xy} 
\end{equation*}
The left vertical arrow is an isomorphism by Proposition \ref{katofinitefieldcase} and the injectivity of the lower horizontal map follows from the vanishing of $H^{2d-1}_{\et}(X,\Q_\ell(d))$. In particular the map $\CH^{d}(X)\{l\}\r H^{2d}_{\et}(X,\Z_\ell(d))$ is injective (see \cite[Thm. 22(iii)]{CSS83} for the surface case).

Now again the fact that the groups 
$$H^{i}_{\et}(X,\Q_\ell/\Z_\ell(r))$$
are finite for $i\neq 2r,2r+1$ and zero for almost all $\ell$ (see \cite[Thm. 2]{CSS82} and \cite[Sec. 2]{CSS83}) implies that $\CH^{d}(X)\{l\}$ is finite and zero for almost all $l$. A similar argument works for the $p$-primary torsion part. In fact there is a surjection $\CH^d(X,1;\Z/p^n\Z)\r \CH^{d}(X)[p^n]$ and $\CH^d(X,q;\Z/p^r\Z)\r H_{\et}^{2d-q}(X,\Z/p^r\Z(d))$
is an isomorphism since the Kato homology groups $KH^{0}_{i}(X,\Z/p^r\Z)$ vanishes for $1\leq i \leq 4$ by \cite[Thm. 0.3]{JS} (see also the proof of \cite[Lem. 6.7]{Lu17'}). Noting that $A_0(X)$ is torsion (see \cite[Prop. 4]{CSS83}), this implies that 
$$A_0(X)$$
is finite. We therefore see that a reduction to surfaces, which was used in \cite{CR86} (see also \cite[Sec. 5]{Ct93}), is not necessary if one uses the Kato conjectures. This was already remarked upon in \cite[p. 25]{Sz10}. 
\end{enumerate}
\end{remark}

\begin{remark}
That $\CH^{d+1}(X,1)$ is torsion is a special case of Parshin's conjecture saying that $K_i(X)\otimes \Q=0$ for a smooth projective scheme over a finite field if $i>0$. Proposition \ref{propositionfinitenesstorsionfinitefield}  says more: $\CH^{d+1}(X,1)$ is not just torsion but finite. This proves a special case of Bass's finiteness conjecture saying that $\CH^r(X,q)$ should be finitely generated for all $r,q\in \N$ if $X$ is of finite type over $\Z$ or a finite field.
\end{remark}

We finish with the following conjecture:
\begin{conj}
Let $X$ be proper and of finite type over $\Z$. Then the group
$$\CH^{d+1}(X,1)$$
is finite.
\end{conj}

\bibliographystyle{plain}
\bibliography{Bibliografie}
\end{document}